\documentclass[12pt,reqno]{amsart}

\usepackage[usenames]{color}
\usepackage[latin1]{inputenc} 

\usepackage[english]{babel} 

\usepackage{enumerate}
\usepackage{amssymb, amsmath}
\usepackage{geometry,graphicx}
\usepackage{hyperref}

\theoremstyle{plain}
\newtheorem{theorem}{Theorem}[section] 

\newtheorem{lemma}[theorem]{Lemma} 

\theoremstyle{definition}
\newtheorem{remark}[theorem]{Remark}

\newtheorem{definition}[theorem]{Definition}

\newcommand{\hooklongrightarrow}{\lhook\joinrel\longrightarrow}  
\DeclareMathOperator{\Gal}{Gal}

\begin{document}

\title{On congruences between normalized eigenforms with different sign at a Steinberg prime\thanks{The first author is partially supported by MICINN grants MTM2015-66716-P. The second author is partially supported by MICINN grant MTM2013-45075-P}
}
\author{Luis Dieulefait}
\thanks{The first author is partially supported by MICINN grants MTM2015-66716-P}
 \author{Eduardo Soto}
\thanks{The second author is partially supported by MICINN grant MTM2013-45075-P}

\address{L. Dieulefait\\ Departament de Matemàtiques i Informàtica\\ Universitat de Barcelona\\ Barcelona, Spain}
\email{ldieulefait@ub.edu}

\address{E. Soto\\
Departament de Matemàtiques i Informàtica\\ Universitat de Barcelona\\ Barcelona, Spain} 
\email{eduard.soto@ub.edu}

\maketitle

\markboth{L. Dieulefait and E. Soto}{On congruent eigenforms with different sign}

\begin{abstract}
Let $f$ be a newform of weight $2$ on $\Gamma_0(N)$ with Fourier $q$-expansion $f(q)=q+\sum_{n\geq 2} a_n q^n$, where $\Gamma_0(N)$ denotes the group of invertible matrices with integer coefficients, upper triangular mod $N$. Let $p$ be a prime dividing $N$ once, $p\parallel N$, a Steinberg prime. Then, it is well known that $a_p\in\{1,-1\}$. We denote by $K_f$ the field of coefficients of $f$. Let $\lambda$ be a finite place in $K_f$ not dividing $2p$ and assume that the mod $\lambda$ Galois representation attached to $f$ is irreducible. In this paper we will give necessary and sufficient conditions for the existence of another Hecke eigenform $f'(q)=q+\sum_{n\geq 2} a'_n q^n$ $p$-new of weight $2$ on $\Gamma_0(N)$ and a finite place $\lambda'$ of $K_{f'}$ such that $a_p=-a'_p$ and the Galois representations $\bar\rho_{f,\lambda}$ and $\bar\rho_{f',\lambda'}$ are isomorphic.
\end{abstract}

%
%

\section{Introduction}
\label{intro}
Let $\bar{\mathbb Q}$ denote the algebraic closure of $\mathbb Q$ in the field $\mathbb C$ of complex numbers. Let $f$ be a cusp Hecke eigenform of weight $2$, level $N$ and trivial nebentypus. We attach to $f$ a sequence $\{a_n\}_{n\geq 1}$ of complex numbers consisting of the Fourier coefficients of $f$ at infinity.  We say that $f$ is normalized if $a_1=1$. In this case $K_f=\mathbb Q(\{a_n\}_n)\subset \bar{\mathbb Q}$ and it is a number field. Let $f$, $f'$ be normalized eigenforms of common weight $2$, level $N$ and $N'$ respectively and trivial nebentypus.  Consider the composite field $L=K_f \cdot K_{f'}$ in $\bar{\mathbb Q}$ and let $\mathfrak l$ be a prime of $L$, $\ell\mathbb Z=\mathfrak l\cap \mathbb Z$. We will be interested in pairs of newforms $f$ and $f'$ for which
\begin{equation}\label{cong'}
a_n \equiv a'_n \pmod {\mathfrak l}\qquad \text{for every integer $n$ coprime to $\ell NN'$}.
\end{equation}
The Fourier coefficients of a newform are completely determined by the $a_p$ coefficients of prime subindex. It is easy to see then that \eqref{cong'} is equivalent to 
\begin{equation}\label{cong}
a_p \equiv a'_p \pmod {\mathfrak l}\qquad \text{for every prime $p\nmid \ell NN'$}.
\end{equation}
Let $\lambda=\mathfrak l \cap \mathcal O_{K_f}$ and $\lambda'=\mathfrak l \cap \mathcal O_{K_f'}$ and assume that the residual Galois representations $\bar\rho_{f,\lambda}$, $\bar\rho_{f',\lambda'}$ attached to $f$ and $f'$ are irreducible. Then $f$, $f'$ satisfy \eqref{cong'} if and only if $\bar\rho_{f,\lambda}$ and $\bar\rho_{f',\lambda'}$ are isomorphic. In general, it is not an easy problem to  find for a given newform $f$ another eigenform $f'$ satisfying \eqref{cong'}, neither proving the existence of such an eigenform $f'$. Ribet's level raising \cite{RibetRai} and level lowering \cite{RibetLow} theorems are very powerful in this context. In this article we will consider a newform $f$ of weight $2$ on $\Gamma_0(N)$ together with a prime $\lambda\nmid 2N$ of $K_f$ and will give necessary and sufficient conditions for the existence of another eigenform $f'$ of weight $2$ on $\Gamma_0(N)$ and a prime $\lambda'$ of $K_{f'}$ such that  
\begin{itemize}
\item $\bar\rho_{f,\lambda}$ and $\bar\rho_{f',\lambda'}$ are isomorphic and
\item $a_p=-a'_p$ for a prime $p$ dividing $N$ once and $f'$ is $p$-new.
\end{itemize}
See \cite{RibetRai} (Ribet 1990) for a definition of $p$-new and theorem \ref{THEtheorem} for the exact statement of the theorem.

\subsection*{Acknowledgements}
The authors are grateful to K. Ribet for providing so much valuable feedback. The second author wants to thank X. Guitart and S. Anni for many stimulating conversations and N. Billerey for helpful comments.

\section{Galois representations mod $\lambda$ attached to a normalized eigenform}
From now on let us fix an odd prime $\ell$ and an immersion $\overline{\mathbb Q}\hooklongrightarrow \overline{\mathbb Q}_\ell$. In particular, for every number field $K$ we have fixed a prime $\lambda$ over $\ell$ and a completion  $K_{\ell}\subset \overline{\mathbb Q}_\ell$ of $K$ with respect to $\lambda$. Let $\overline{\mathbb F}_\ell$ denote the residue field of the ring of integers of $\overline{\mathbb Q}_\ell$, which is indeed an algebraic closure of $\mathbb F_\ell$.
Let $f$ be a normalized eigenform on $\Gamma_0(N)$ with $q$-expansion at infinity
$
f(z) = q + \sum_{n\geq 2} a_n q^n.
$
As in the introduction, $K_f$ denotes   the number field $\mathbb Q(\{a_n\}_n)$ of coefficients of $f$ and by $\mathcal O_f$ its ring of integers. Consider the $\ell$-adic Galois representation attached to $f$ by Deligne
$$
\rho_{f,\ell}: \Gal(\overline{\mathbb Q}\mid \mathbb Q)\longrightarrow GL_2(K_{f,\ell})
$$
where $K_{f,\ell}$ denotes the completion of $K_f$ with respect to (the fixed prime $\lambda$ above) $\ell$ and $\mathcal O_{f,\ell}$ denotes its ring of integers. Since $\Gal(\overline{\mathbb Q}\mid \mathbb Q)$ is compact one can (non-canonically) embed its image in  $GL_2(\mathcal O_{f,\ell})$,
$$
\iota : \text{Im}\,\rho_{f,\ell}\longrightarrow GL_2(\mathcal O_{f,\ell}).
$$
 Recall that $\mathcal O_{f,\ell}$ is a local ring whose residue field $\mathcal O_{f,\ell}/\mathfrak m_{f,\ell}$ is a finite extension of $\mathbb F_\ell$ contained in $\overline{\mathbb F}_\ell$. Reducing $\iota\circ \rho_{f,\ell}$ mod $\mathfrak m_{f,\ell}$, we obtain the mod $\ell$ Galois representation
$$
\bar\rho_{f,\ell}:\Gal(\overline{\mathbb Q}\mid\mathbb Q)\longrightarrow GL_2(\mathcal O_{f,\ell}/\mathfrak m_{f,\ell})\hooklongrightarrow GL_2(\overline{\mathbb F}_\ell)
$$
attached to $f$ (independent of $\iota$ up to semi-simplification), satisfying
$$
\text{tr}\,\bar\rho_{f,\ell} (Frob_p) \equiv a_p\qquad \det\bar\rho_{f,\ell}(Frob_p)\equiv p \pmod {\mathfrak m_{f,\ell}}
$$
for every rational prime $p\nmid \ell N$. 

\section{ Lowering and raising the levels}
We state here Ribet's theorems, they will be the main tools needed in the proof of our theorem. See \cite{RibetLow} and \cite{RibetRai} (theorem 1 and remarks in section $3$) for the proofs.
\begin{theorem}[Ribet's level lowering theorem]
Let $f$ be a newform of weight $2$ on $\Gamma_0(N)$, let $p$ be a prime dividing $N$ once. Assume that $\ell\nmid2p$ and that the mod $\ell$ Galois representation 
$$
\bar\rho_{f,\ell}:\Gal(\overline{\mathbb Q}\mid \mathbb Q)\longrightarrow GL_2(\overline{\mathbb F}_\ell)
$$
is irreducible and unramified at $p$. If one or both of the following conditions hold
\begin{enumerate}[1.]
\item $\ell\nmid N$,
\item $\bar\rho_{f,\ell}\vert_{G_{\mathbb Q(\zeta_\ell)}}$ is irreducible,
\end{enumerate}
then there exists a newform $f'$ of weight $2$ on $\Gamma_0(M)$ for some $M$ divisor of $N/p$ such that $\bar\rho_{f',\ell}$ is isomorphic to $\bar\rho_{f,\ell}$.
\end{theorem}
\begin{proof}
By theorem 1.1 of \cite{RibetLow} we may assume that $\bar\rho_{f,\ell}\vert_{G_{\mathbb Q(\zeta_\ell)}}$ is irreducible. We first prove the existence of a representation $\rho:\Gal(\bar{\mathbb Q}\mid \mathbb Q)\rightarrow GL_2(\bar{\mathbb Q}_p)$ lifting $\bar\rho_{f,\ell}$ satisfying enough properties such that, if modular, it arises from a cusp eigenform of weight $2$ on $\Gamma_0(N/p)$. We use some results of \cite{Snowden} \textsection 2, \textsection 7 and we follow the notation therein. Consider the lifting problem with 
\begin{enumerate}[(i)]
\item $\Sigma$ equal to the set of primes different from $p$ at which $\rho_{f,\ell}$ is ramified and $\ell$,
\item $\psi$ trivial,
\item type function $t$ equal to the one attached to $\rho_{f,\ell}$,
\item for each $q\in\Sigma$ the inertial type $\tau_q$ of $\rho_{f,\ell}\vert_{G_q}$.
\end{enumerate}
By proposition 2.6.1 in \cite{Snowden}, $t$ is definite. Since $\{\rho_{f,\ell}\vert_{G_q}\}_{q\in \Sigma}$ is a local solution to our lifting problem theorem 7.2.1 in \cite{Snowden} says that there is a global solution $\rho$ that is finitely ramified weight two. It is irreducible since $\bar\rho$ is irreducible, and odd since its determinant is the cyclotomic character. 
By theorem 1.1.4 in \cite{Snowden} $\rho$ is modular, so it arises from a newform $g$ of weight two. Comparing $\rho_{f,\ell}\vert_{G_q}$ and $\rho\vert_{G_q}$ at every $q$ we have that $g$ has level $N/p$ and trivial nebentypus (since $\psi=1$).
\end{proof}
\begin{theorem}[Ribet's level raising theorem]
Let $f$ be a 
normalized eigenform
 of weight $2$ on $\Gamma_0(N)$ such that the mod $\ell$ Galois representation
$$
\bar\rho_{f,\ell}: \Gal(\overline{\mathbb Q}\mid \mathbb Q)\longrightarrow GL_2(\overline{\mathbb F}_\ell)
$$
is irreducible. Let $p\nmid \ell N$ be a prime satisfying
$$
\emph{\text{tr}}\,\bar\rho_{f,\ell}(Frob_p) \equiv (-1)^j(p+1) \pmod \ell
$$
for some $j\in\{0,1\}$. 
Then there exists a normalized eigenform $f'$ $p$-new of weight $2$ on $\Gamma_0(pN)$ such that $\bar\rho_{f,\ell}$ is isomorphic to $\bar\rho_{f',\ell}$. Moreover, the $p$-th Fourier coefficient $a'_p$ of $f'$ equals $(-1)^j$. If $p+1\equiv 0\pmod \ell$ then there are at least two such eigenforms $f'$: one for each coefficient $a'_p\in\{\pm1\}$.
\end{theorem}

\section{Local decomposition at Steinberg primes and proof of the theorem}
In this article we consider newforms $f$ of weight $2$ on $\Gamma_0(N)$ and we work with primes $p$ dividing $N$ once. Recall that in this case the local type of (the automorphic form attached to) $f$ at $p$ is a twist of the Steinberg representation. 
\begin{definition}\label{defsteinberg}
Let $f$ be a normalized cusp Hecke eigenform of weight $2$ level $N$ and trivial nebentypus. We say that a prime $p$ is a \emph{Steinberg prime} of $f$ if $p\parallel N$ and $f$ is $p$-new.
\end{definition}
	\begin{remark}
	With the hypothesis of definition \ref{defsteinberg}, corollary 4.6.20 in \cite{Miyake} says that there is a unique newform $g$ of weight $2$ on $\Gamma_0(N')$  for some divisor $N'$ of $N$, $p\mid N'$, such that $f$ is in the old-space generated by $g$. Theorem 4.6.17 in \cite{Miyake} implies that $a_p(g)\in \{-1,1\}$. Since $f$ is normalized and $p\nmid N/N'$ it is easy to see that $a_p(f)=a_p(g)$.
	\end{remark}

Here we state a useful lemma related to the local behavior of the mod $\ell$ Galois representations at a Steinberg prime. 

\begin{lemma}[Langlands]\label{locallemma}
Let $f$ be a cusp Hecke eigenform of weight $2$, level $N$ and trivial nebentypus. Let $p$ be a Steinberg prime of $f$. Then one has that
$$
\bar\rho_{f,\ell}\vert_{D_p} \sim 
\left(
\begin{array}{cc}
\bar\chi\cdot\bar\varepsilon_\ell  &*\\
0	&\bar\chi
\end{array}
\right)
$$
for every prime $\ell\neq p$, where $\bar\chi:D_p\rightarrow \mathbb F^*_\ell$ denotes the unramified character that maps $\text{Frob}_p$ to $a_p(f)$ and $\bar\varepsilon_\ell$ denotes the mod $\ell$ cyclotomic character.
\end{lemma}

\begin{proof}
See \cite{LW} (Loeffler, Weinstein 2012)  proposition 2.8, we follow the notation therein. The newform $f$ is $p$-primitive since $p\parallel N$. Recall that Hecke correspondence (a modification of local Langlands correspondence) attaches to $\pi_{f,p}\simeq St\otimes \alpha$, for an unramified character $\alpha$ of $\mathbb Q_p^*$, a two dimensional Weil-Deligne representation that corresponds to a Galois representation $r:\Gal(\bar{\mathbb Q}_p\mid\mathbb Q_p)\rightarrow GL_2(\bar{\mathbb Q}_\ell)$ of the form
$$
r\sim
\left(\begin{array}{cc}
\alpha \cdot \varepsilon_\ell  	&*\\
0   						&\alpha
\end{array}
\right),
$$
where we identify $\alpha$ with a character of $\Gal(\bar{\mathbb Q}_p\mid \mathbb Q_p)$ via local class field theory. It is a theorem of Carayol \cite{Carayol} that $\rho_{f,\ell}\vert_{D_p}$ and $r$ are isomorphic.
\end{proof}
\begin{theorem}[Main result]\label{THEtheorem}
Let $f$ be a newform of weight $2$ on $\Gamma_0(N)$ and let $p$ be a prime at which $f$ is Steinberg. Let $\ell \nmid 2p$ be a prime.
Assume either 
\begin{enumerate}[1.]
\item $\ell \nmid N$ and $\bar\rho_{f,\ell}$ is absolutely irreducible, or
\item $\bar\rho_{f,\ell}\vert_{G_{\mathbb Q(\zeta_\ell)}}$ is absolutely irreducible.
\end{enumerate}
Then the following are equivalent:
\begin{enumerate}[(a)]
\item $\bar{\rho}_{f,\ell}$ is unramified at $p$ and 
$$
p \equiv -1 \pmod \ell.
$$
\item There is a normalized eigenform $f'$ $p$-new of weight $2$ on $\Gamma_0(N)$ such that $\bar\rho_{f',\ell}$ is isomorphic to $\bar\rho_{f,\ell}$ and 
$$
a_p = -a'_p
$$
where $a_p$ (resp. $a'_p$) is the $p$-th Fourier coefficient of $f$ (resp. $f'$).
\end{enumerate}
\end{theorem}
\begin{proof}
\emph{(b) implies (a)}

Let us write $\bar\rho = \bar\rho_{f,\ell}$ and $\bar{\rho}'=\bar\rho_{f',\ell}$ for simplicity. Let $D_p\subseteq \Gal(\overline{\mathbb Q}\mid\mathbb Q)$ be a decomposition group of $p$. Since $\{a_p,a'_p\} = \{1, -1\}$, we may assume without loss of generality that  $\bar\rho$, $\bar\rho'$ act locally at $p$ as
$$
\bar\rho \vert_{D_p}\sim
\left(
\begin{array}{cc}
\bar\varepsilon_\ell  & *\\
0 		  &1
\end{array}
\right)
$$
and 
$$
\bar\rho' \vert_{D_p}\sim
\left(
\begin{array}{cc}
\bar\chi\cdot\bar\varepsilon_\ell  & *\\
0 		  &\bar\chi
\end{array}
\right)
$$
due to lemma \ref{locallemma}.
Since $\bar\rho$ and $\bar\rho'$ are isomorphic so are their local behaviors. Thus, specializing at  a Frobenius map
$$
\left(
\begin{array}{cc}
p  & *\\
0 		  &1
\end{array}
\right)
\sim
\left(
\begin{array}{cc}
-p  & *\\
0 		  &-1
\end{array}
\right)\pmod \ell.
$$
Eigenvalues must coincide and $\ell>2$ so
$$
p\equiv -1 \pmod \ell
$$
To see that $\bar\rho$ is unramified at $p$ notice that 
$$
\bar\chi \bar\varepsilon_\ell\not\equiv \bar\varepsilon_\ell.
$$
Thus, $\bar\rho\vert_{D_p}\simeq \bar\varepsilon_\ell \oplus\bar\chi\bar\varepsilon_\ell$. Indeed, lemma \ref{locallemma} and the isomorphism $\bar\rho\simeq\bar\rho'$ say that $\bar\rho\vert_{D_p}$ has two $1$-dimensional subrepresentations. Since the actions are different, they generate the whole space. Hence $\bar\rho$ is unramified at $p$.

\bigskip
\noindent\emph{(a) implies (b)}\\
Ribet's lowering level theorem applies to the modular representation $\bar\rho_{f,\ell}$ of level $N$. Thus, there exists a newform $g$ of weight $2$ on $\Gamma_0(M)$, for some $M\mid N/p$ such that $\bar\rho_{g,\ell}\sim \bar\rho_{f,\ell}$. Moreover, we have that
\begin{align*}
\text{tr} \bar\rho_{f,\ell}(\text{Frob}_p)&\equiv a_p\cdot (p+1)\\
							&\equiv 0  \pmod \ell
\end{align*}
by lemma \ref{locallemma}.
Now we can apply Ribet's raising level theorem to $\bar\rho_{g,\ell}$ and there exists an eigenform $f'$ $p$-new on $\Gamma_0(N)$ such that 
$$
\bar\rho_{f',\ell}\sim\bar\rho_{f,\ell}.
$$  
By \S $3$ page $9$ of \cite{RibetRai}, when both conditions 
$$
\text{tr}\bar \rho_{g,\ell} ( Frob_p )\equiv \pm (p+1) \pmod \ell
$$
are satisfied Ribet's proof allows us to choose the $a_p$ coefficient of $f'$. We shall choose $f'$ such that $a'_p = - a_p$ and the implication holds.
\end{proof}
\begin{remark}\label{remark:new}
Let $f$ be a newform satisfying $(a)$. We expect that theorem \ref{THEtheorem} can be strengthened so that $f'$ can be chosen to be a newform, not necessarily unique. This would follow from a stronger version of \cite{RibetRai}. See section $5$ for an example.
\end{remark}
\section{An example}
In this section we are going to give an example of mod $5$ Galois representation to which our theorem applies. We will use many well-known properties of elliptic curves without proof.

Let $\ell$ be a prime, $n>0$ an integer and $E$ an elliptic curve over $\mathbb Q$. The $\ell^n$-th torsion group $E[\ell^n]$ of $E$ has a natural structure of free $\mathbb Z/\ell^n\mathbb Z$-module of rank $2$. 
The action of the Galois group $G=\Gal(\overline{\mathbb Q}\mid\mathbb Q)$ is compatible with the $\mathbb Z/\ell^ n \mathbb Z$-module structure of  $E[\ell^n]$, so that $E[\ell^n]$ has a natural structure of 
$\mathbb (\mathbb Z/\ell^n\mathbb Z)[G]$-module.  
That is, the action induces a group homomorphism
$$
\bar\rho_{E,\ell^n}: \Gal(\overline{\mathbb Q}\mid \mathbb Q)\longrightarrow Aut_{\mathbb Z/\ell^n\mathbb Z} (E[\ell^n])\simeq GL_2(\mathbb Z/\ell^n\mathbb Z)
$$
where the isomorphism depends on the choice of a basis in $E[\ell^n]$. The case $n=1$ is of special interest and is known as the mod $\ell$ Galois representation attached to $E$. The Tate module $\mathcal T_\ell\, E=\varprojlim  E[\ell^n]$ of $E$ at $\ell$ is a free $\mathbb Z_\ell$-module of rank $2$. The morphisms $\{\bar\rho_{E,\ell^n}\}$ induce a group morphism
$$
\rho_{E,\ell}:\Gal(\overline{\mathbb Q}\mid \mathbb Q)\longrightarrow Aut (\mathcal T_\ell \, E)\simeq GL_2(\mathbb Z_\ell)
$$
known as the $\ell$-adic Galois representation attached to $E$. The mod $\ell$ Galois representation $\bar\rho_{E,\ell}$ attached to  $E$ can be recovered from $\rho_{E,\ell}$ by taking reduction mod $\ell\mathbb Z_\ell$.

Well-known modularity
 theorems as Wiles, Taylor-Wiles and Breuil-Conrad- Diamond-Taylor state that such a $\ell$-adic Galois representation is isomorphic to the Galois representation $\rho_{f_E,\ell}$ attached
 to some newform $f_E$ of weight $2$ on $\Gamma_0(N)$ for $N$ equal to the conductor of $E$ and $K_f=\mathbb Q$.
 Moreover the $p$-th Fourier coefficient of 
 $f_E$ coincides with the $c_p$ coefficient of $E$ (defined below) for every prime $p$. In this section we will apply theorem \ref{THEtheorem} to the mod $5$ Galois representation $\bar\rho_{E,5}\simeq \bar\rho_{f_E,5}$ attached to  the following elliptic curve given by a (global minimal) Weierstrass equation:
$$
E:ZY^2 +XYZ + YZ^2 = X^3 + X^2Z -614X Z^2-5501Z^3.
$$
Its discriminant is
$$
\Delta = 2^5\cdot19^5\cdot 37\\
$$
For every prime $p$ let $\tilde E_p$ denote the curve obtained by reducing mod $p$ a global minimal Weierstrass model of $E$. As usual, we consider the value
$$
c_p=p+1 -\#\tilde E_p
$$
for every prime $p$.
One can check that 
$$
\begin{cases}
\text{$\tilde E_2$ has a node whose tangent lines are defined over $\mathbb F_2$,}\\
\text{$\tilde E_p$ has a node whose tangent lines have slopes in $\mathbb F_{p^2}\setminus \mathbb F_p$  for $p\in \{19,37\}$,}\\
\text{$\tilde E_p$ is an elliptic curve over $\mathbb F_p$, otherwise.}
\end{cases}
$$
Thus $E$ has conductor $N= 2\cdot 19\cdot 37=1406$, $c_2 = -c_{19}=-c_{37}=1$ and there is a newform $f=q+\sum_{n\geq 2} a_n q^n$ of weight $2$ on $\Gamma_0(1406)$ such that 
\begin{itemize}
\item $\rho_{E,5}$ and $\rho_{f,5}$ are isomorphic.
\item
$c_p = a_p$ for every prime $p$.
\end{itemize}
Let us see now that $\bar\rho_{f,5}$ satisfies the hypothesis of the theorem \ref{THEtheorem} for $p=19$. Since
\begin{align*}
E_3&=\{[x:y:z]\in \mathbb P^2_{\mathbb F_3}: zy^2 + xyz+ yz^2 = x^3+x^2z+xz^2+z^3\}\\
	&=\{[0:1:0], [2:0:1]\}
\end{align*}
then $c_3= 2$ and the characteristic polynomial of $\bar\rho_{f,5}(Frob_3)$ is congruent to
$$
P(X) = X^2- 2X + 3   \pmod 5.
$$
Since the discriminant of $P(X)$ is $2\pmod 5$ and $2$ is not a square in $\mathbb F_5$, then $P(X)$ is irreducible over $\mathbb F_5$. In particular $\bar\rho_{f,5}:\Gal(\overline{\mathbb Q}\mid\mathbb Q)\longrightarrow GL_2(\mathbb F_5)$ is irreducible. It is well known that such a representation is irreducible if and only if it is absolutely irreducible, thus $\bar\rho_{f,5}$ is absolutely irreducible. Indeed, $\bar\rho_{f,5}$ is odd and hence its image contains a matrix $E$ with eigenvalues $\{\pm 1\}$. Say $\bar\rho_{f,5}(c)=  E$. If $\bar\rho_{f,5}$ is not absolutely irreducible $\bar\rho_{f,5}$ is conjugate over $\bar{\mathbb F}_5$ to a representation of the form
$$
\bar r=\left(
\begin{array}{cc}
\theta_1  &*\\
0 		&\theta_2
\end{array}
\right)
$$
for some characters $\theta_1,\theta_2$. Take $\epsilon:=\theta_1(c)\in\{\pm1\}$.
By changing coordinates over $\mathbb F_5$ we may assume that
$$
\bar\rho_{f,5}(c)=\left(\begin{array}{cc}
 \epsilon &0\\
 0   &-\epsilon
\end{array}\right).
$$ By comparing $\bar r$ and $\bar\rho_{f,5}$ at $c$ it follows that the conjugation matrix is upper triangular and hence that $\bar\rho_{f,5}$ is upper triangular over $\mathbb F_5$.

On the other hand, it is well known (see \cite{DDT} proposition 2.12) that for a prime $p\neq \ell$ dividing once the conductor of an elliptic curve $E$, $\bar\rho_{E,\ell}$ is  unramified at $p$ if and only if $\ell\mid v_{p}(\Delta)$. Since
$
v_{19}(\Delta)=5,
$
$\bar\rho_{f,5}$ is also unramified at $19$. Hence theorem \ref{THEtheorem} applies to $\bar\rho_{E,5}$ and there exists another eigenform $f'$ of weight $2$ on $\Gamma_0(1406)$ 
such that $\bar\rho_{f',\ell}$ is isomorphic to $\bar\rho_{E,5}$. Thus, 
$$
c_p \equiv a_p \pmod{ \mathfrak{m}_{f',\ell}}.
\qquad\text{for every $p\nmid 1406\cdot 5$}.
$$
and the $19$th Fourier coefficient $a'_{19}$ of $f'$ satisfies
$$
a'_{19}=-c_{19}=1.
$$

In this example we can find an elliptic curve $E'$ of conductor $1406$ such that its corresponding newform $f'$ satisfies part (b) of theorem \ref{THEtheorem}. In order to find $E'$ we have assumed that 
$f'$ can be chosen to be a newform,
see remark \ref{remark:new}. Assuming that $f'$ is newform we can determine $a'_2$ and $a'_{37}$ of $f'$ as follows. If $a'_2=-c_2$ theorem \ref{THEtheorem} applies with $\ell=5$ and $p=2$ so we conclude that $2\equiv -1\pmod 5$. Hence, $a'_2\neq -c_2$. Since $a'_2\in\{\pm 1\}$ then $a'_2 = c_2=1$. Similarly one can prove that $a'_{37}= c_{37}=-1$. There are three newforms (up to conjugation) with this configuration of signs $(a'_2,a'_{19},a'_{37})=(1,1,-1)$ out of sixteen newforms of weight $2$ on $\Gamma_0(1406)$. One of those (the only one satisfying $a'_3\equiv c_3 \pmod{ \mathfrak m_{f',5}}$) corresponds to the elliptic curve $E'$ given by a global minimal Weierstrass equation
$$
E': Y^2 Z + XYZ + YZ^2 = X^3- X^2 Z-1191X Z^2 +507615 Z^3.
$$
One can check that its conductor is $N= 1406$ and that $c'_2 = c'_{19} = -c'_{37}=1$. We will use Sturm's bound (see theorem 9.18 in \cite{Steinbook} Stein's book) in order to prove that $E'$ corresponds to an eigenform $f'$ as in theorem \ref{THEtheorem}. Notice that 
$$
c'_{19} =19+1-\# \tilde {E'}= 1 = - c_{19}.
$$
Hence, Sturm's result does not apply to the pair $(f_{E}, f_{E'})$. After twisting both modular forms $f_E$, $f_{E'}$ by the quadratic character $\psi$ of conductor $19$ we get two cusp forms $f^\psi_E$, $f^\psi_{E'}$ of weight $2$ on $\Gamma_0(1406\cdot 19)$. Sturm's bound for cusp forms of weight $2$ on  $\Gamma_0(1406\cdot 19)$ is $7218.38$ and one can check computationally that 
\begin{equation}\label{congruence}
\psi(p)c_p \equiv \psi(p)c'_p\pmod 5
\end{equation}
for every prime $p< 7219$. Hence, Sturm's result applies to $(f^{\psi}_E,f^{\psi}_{E'})$ and \eqref{congruence} is true for every $p$. Thus
$$
c_p\equiv c'_p \pmod 5
$$
for every prime $p$ for which $\psi(p)\neq 0$, i.e for every prime $p\neq 19$.
\begin{remark}
Take again $\ell=5$, $p=19$. Similarly one can check that at level $741$ there is an elliptic curve (Cremona's labeling 741b) satisfying theorem \ref{THEtheorem} a). Notice that there exists no elliptic curve of conductor $741$ corresponding to an $f'$ as in theorem \ref{THEtheorem} b). However, it does exist a newform $f$ on $\Gamma_0(741)$ and weight $2$ satisfying theorem \ref{THEtheorem} b), it satisfies $[K_f:\mathbb Q]=4$.
\end{remark}



\end{document}